\documentclass[11pt]{amsart}

\usepackage{amsmath,amsthm,amssymb,amsfonts}
\usepackage{enumitem}
\usepackage{aliascnt}
\usepackage[colorlinks=true,linkcolor=blue,citecolor=blue,urlcolor=blue]{hyperref}
\usepackage[nameinlink,noabbrev]{cleveref}

\newtheorem{theorem}{Theorem}[section]

\newaliascnt{lemma}{theorem}
\newtheorem{lemma}[lemma]{Lemma}
\aliascntresetthe{lemma}

\newaliascnt{corollary}{theorem}
\newtheorem{corollary}[corollary]{Corollary}
\aliascntresetthe{corollary}

\newaliascnt{proposition}{theorem}
\newtheorem{proposition}[proposition]{Proposition}
\aliascntresetthe{proposition}

\theoremstyle{definition}

\newaliascnt{definition}{theorem}
\newtheorem{definition}[definition]{Definition}
\aliascntresetthe{definition}

\newaliascnt{remark}{theorem}
\newtheorem{remark}[remark]{Remark}
\aliascntresetthe{remark}

\newaliascnt{example}{theorem}

\aliascntresetthe{example}

\crefname{theorem}{Theorem}{Theorems}
\Crefname{theorem}{Theorem}{Theorems}
\crefname{lemma}{Lemma}{Lemmas}
\Crefname{lemma}{Lemma}{Lemmas}
\crefname{proposition}{Proposition}{Propositions}
\Crefname{proposition}{Proposition}{Propositions}
\crefname{corollary}{Corollary}{Corollaries}
\Crefname{corollary}{Corollary}{Corollaries}
\crefname{definition}{Definition}{Definitions}
\Crefname{definition}{Definition}{Definitions}
\crefname{remark}{Remark}{Remarks}
\Crefname{remark}{Remark}{Remarks}
\crefname{example}{Example}{Examples}
\Crefname{example}{Example}{Examples}
\crefname{equation}{equation}{equations}
\Crefname{equation}{Equation}{Equations}

\DeclareMathOperator{\rank}{rank}
\DeclareMathOperator{\Roots}{Roots}
\DeclareMathOperator{\Harm}{Harm}

\newcommand{\R}{\mathbb R}
\newcommand{\Z}{\mathbb Z}
\newcommand{\cL}{\mathcal L}
\newcommand{\cC}{\mathcal C}
\newcommand{\ip}[2]{\langle #1,#2\rangle}
\newcommand{\norm}[1]{\lVert #1\rVert}
\newcommand{\Sk}{\mathcal S_k}
\newcommand{\Xk}{X_k}

\setlist[enumerate,1]{label=\textup{(\roman*)}, ref=\textup{(\roman*)}}
\setlist[enumerate,2]{label=\textup{(\alph*)}, ref=\textup{(\alph*)}}

\newlist{theoremcases}{enumerate}{1}
\setlist[theoremcases]{
  label=\textup{(\roman*)},
  ref=\textup{(\roman*)},
  leftmargin=*,
  itemsep=.3em
}

\crefname{enumi}{part}{parts}
\Crefname{enumi}{Part}{Parts}
\crefname{theoremcasesi}{part}{parts}
\Crefname{theoremcasesi}{Part}{Parts}

\title[Equality in a Reverse Minkowski Shell Bound]{Equality in a Reverse
Minkowski Shell Bound for Integral Lattices via Spherical Designs}

\author[Scott Duke Kominers]{Scott Duke Kominers}

\address{Harvard Business School; Department of Economics and Center of
Mathematical Sciences and Applications, Harvard University; and a16z crypto}

\email{kominers@fas.harvard.edu}

\thanks{I used LLMs to assist with computations, analysis, and synthesis in
the preparation of this article, particularly GPT-5.5 Pro and Claude 4.7 Opus
(accessed in part via Poe with the support of Quora, where I am an advisor).
I particularly appreciate a thorough review from Refine.ink. The problem,
methods, and eventual written form are my own; and of course any errors remain
my responsibility. This work was conducted while I was visiting the
Technological Innovation, Entrepreneurship, and Strategic Management (TIES)
Group at the MIT Sloan School of Management; I greatly appreciate their
hospitality.}


\begin{document}

\begin{abstract}
For a full-rank integral lattice \(\cL\subset\R^n\), Regev and
Stephens-Davidowitz proved that
\[
  N_{=k}(\cL):=|\{y\in\cL:\norm{y}^2=k\}|
  \le 2\binom{n+2k-2}{2k-1}.
\]
We classify the equality cases.  For \(n\ge2\), equality holds if and only if
either \(k=1\) and \(\cL\cong\Z^n\), or \(n=8\), \(k=2\), and
\(\cL\cong E_8\).  For \(n=1\), equality holds exactly when \(\cL\) represents
\(k\).

The proof shows that equality is rigid.  Saturation of the shell bound forces
the normalized norm-\(k\) shell to be an antipodal tight spherical
\((4k-1)\)-design.  The associated Delsarte--Goethals--Seidel annihilator
polynomial gives an arithmetic root condition, which isolates \(E_8\) at
\(k=2\), rules out \(k=3\), and combines with the
Bannai--Damerell/Bannai theorem and an elementary circle
argument to exclude all remaining cases in dimension at least \(2\).
\end{abstract}

\maketitle

\section{Introduction}

Minkowski's convex body theorem gives a fundamental lower-bound principle in
the geometry of numbers: a sufficiently large centrally symmetric convex body
must contain a nonzero point of a lattice.  ``Reverse Minkowski'' questions ask
for complementary upper bounds on how many lattice points of prescribed length,
or in prescribed regions, can occur under arithmetic or stability hypotheses.

Let \(\cL\subset \R^n\) be an integral lattice, meaning that
\[
  \ip{x}{y}\in \Z \qquad \text{for all }x,y\in \cL.
\]
For a positive integer \(k\), write
\[
  \Sk(\cL):=\{y\in \cL:\norm{y}^2=k\},
  \qquad
  N_{=k}(\cL):=|\Sk(\cL)|.
\]
Regev and Stephens-Davidowitz \cite{RSD2026} proved that
\begin{equation}\label{eq:RSDbound}
  N_{=k}(\cL)\le 2\binom{n+2k-2}{2k-1}.
\end{equation}
Their proof applies the antipodal, or line-system, form of the
Delsarte--Goethals--Seidel absolute bound to the normalized shell
\[
  \Xk(\cL):=\frac{1}{\sqrt{k}}\Sk(\cL)\subset S^{n-1}.
\]
Indeed, if \(y,z\in\Sk(\cL)\), then
\[
  \left\langle \frac{y}{\sqrt{k}},\frac{z}{\sqrt{k}}\right\rangle
  =
  \frac{\ip{y}{z}}{k}\in \frac1k\Z.
\]
Thus the set of distinct inner products, including the antipodal value
\(-1\), has size at most \(2k\).  The antipodal
Delsarte--Goethals--Seidel absolute bound,
\[
    |X|\le 2\binom{n+s-2}{s-1}
\]
for an antipodal set \(X\) with \(s=|A(X)|\) distinct inner products, then
yields \eqref{eq:RSDbound} on substituting \(s\le 2k\) and using monotonicity
of \(2\binom{n+m-1}{m}\) in \(m\).

In this paper, we determine when \eqref{eq:RSDbound} is sharp.
The answer is that equality is extremely rare.

\begin{theorem}[Main classification]\label{thm:main}
Let \(\cL\subset \R^n\) be a full-rank integral lattice, \(n\ge 1\), and let
\(k\ge 1\) be an integer.  Equality holds in \eqref{eq:RSDbound}, i.e.,
\[
  N_{=k}(\cL)= 2\binom{n+2k-2}{2k-1},
\]
if and only if one of the following mutually exclusive cases occurs:
\begin{theoremcases}
\item\label{main:rank-one}
\(n=1\), say \(\cL=a\Z\) with \(a>0\) and \(a^2\in\Z\), and there exists
\(m\in\Z\) with \(a^2m^2=k\); equivalently, the one-dimensional lattice
\(\cL\) represents \(k\).
\item\label{main:k-one}
\(n\ge 2\), \(k=1\), and \(\cL\) is isometric to \(\Z^n\).
\item\label{main:e-eight}
\(n=8\), \(k=2\), and \(\cL\) is isometric to the \(E_8\) root lattice.
\end{theoremcases}
\end{theorem}

Thus, in rank at least \(2\), the only sharp shells are the coordinate
minimal shell of the cubic lattice and the root shell of \(E_8\).  The
rank-\(1\) family appears because in dimension \(1\) the right-hand side of
\eqref{eq:RSDbound} is always \(2\).

The main conceptual observation is that equality in \eqref{eq:RSDbound} is not
merely a numerical endpoint---it forces a strong geometric structure.  The
Delsarte--Goethals--Seidel equality theorem implies that, for \(n\ge2\),
a sharp normalized shell is an antipodal tight spherical \((4k-1)\)-design.
This strengthening is crucial: For instance, when \(k=2\), equality forces
a tight spherical \(7\)-design, not merely a spherical \(3\)-design; the
latter condition alone would be far too weak to isolate \(E_8\).

The equality theorem also determines the annihilator polynomial of the
inner-product set.  In the present arithmetic setting this yields a simple
integrality filter: the roots of a certain universal Gegenbauer sum must be
exactly
\[
  0,\quad \pm\frac1k,\quad \pm\frac2k,\quad \ldots,\quad \pm\frac{k-1}{k}.
\]
For \(k=2\), the filter forces \(n=8\); for \(k=3\), it gives an immediate
contradiction; and for \(k\ge4\), equality would produce a tight spherical
design of strength at least \(15\), which is ruled out in dimensions at least
\(3\) by the Bannai--Damerell and Bannai nonexistence theorem.  The remaining
circle case is elementary.

The paper is organized as follows.  \Cref{sec:prelim}
records the needed facts about lattice shells, Delsarte--Goethals--Seidel
bounds, and tight spherical designs.  \Cref{sec:equality-dgs} proves that
sharpness of \eqref{eq:RSDbound} forces a tight \((4k-1)\)-design and derives
the annihilator identity.  \Cref{sec:integrality-filter} extracts the
integrality filter and performs the \(k=2\) and \(k=3\) calculations.
\Cref{sec:classification-proof} proves \cref{thm:main}.
\Cref{sec:verification} checks the explicit equality examples and gives
comparison examples involving the circle and the Leech lattice.
\Cref{sec:remarks} records some concluding comments and related questions.

\section{Preliminaries}\label{sec:prelim}

\subsection{Integral lattices and shells}

Throughout, a lattice \(\cL\subset\R^n\) is assumed to have full rank unless
explicitly stated otherwise.  It is called \emph{integral} if
\[
  \ip{x}{y}\in\Z
  \qquad
  \text{for all $x,y\in\cL$}.
\]
For \(k\in\Z_{>0}\), its \emph{norm-\(k\) shell} is
\[
  \Sk(\cL)=\{y\in\cL:\norm{y}^2=k\}.
\]
The \emph{normalized shell} is
\[
  \Xk(\cL)=\frac{1}{\sqrt{k}}\Sk(\cL)\subset S^{n-1},
\]
provided \(\Sk(\cL)\neq\varnothing\).

Every normalized lattice shell is antipodal: if \(y\in\Sk(\cL)\), then
\(-y\in\Sk(\cL)\).

\begin{definition}[Antipodal set]
A finite set \(X\subset S^{n-1}\) is \emph{antipodal} if \(X=-X\), or
equivalently if \(X\) is a union of antipodal pairs \(\{x,-x\}\).
\end{definition}

\subsection{Inner-product sets}

For a finite set \(X\subset S^{n-1}\), define its \emph{inner-product set} by
\[
  A(X):=\{\ip{x}{y}:x,y\in X,\ x\ne y\}.
\]
When \(X\) is antipodal, the set \(A(X)\cup\{1\}\) is invariant under
multiplication by \(-1\).

For a normalized lattice shell \(X=\Xk(\cL)\), if
\(x=y/\sqrt{k}\) and \(x'=z/\sqrt{k}\) are distinct elements of \(X\), then
\[
  \ip{x}{x'}=\frac{\ip{y}{z}}{k}\in\frac1k\Z.
\]
By Cauchy--Schwarz, \(\ip{y}{z}=k\) only when \(y=z\), and
\(\ip{y}{z}=-k\) only when \(y=-z\).  Hence
\begin{equation}\label{eq:shell-inner-products}
  A(\Xk(\cL))
  \subseteq
  \{-1\}\cup
  \left\{\frac{j}{k}: -k<j<k,\ j\in \Z\right\};
\end{equation}
including the antipodal value \(-1\), the set has at most \(2k\) elements, and
the non-antipodal part has at most \(2k-1\) elements.

\subsection{Spherical designs and Gegenbauer sums}

We use the standard definition of a spherical design.

\begin{definition}[Spherical design]
A finite nonempty set \(X\subset S^{n-1}\) is a \emph{spherical \(t\)-design}
if
\[
  \frac1{|X|}\sum_{x\in X} f(x)
  =
  \int_{S^{n-1}} f(u)\,d\sigma(u)
\]
for every polynomial \(f\) on \(\R^n\) of total degree at most \(t\), where
\(\sigma\) is normalized rotation-invariant measure on \(S^{n-1}\).

The \emph{strength} of \(X\) is the largest integer \(t\) for which \(X\) is
a spherical \(t\)-design.
\end{definition}

Equivalently, \(X\) is a spherical \(t\)-design if
\[
  \sum_{x\in X} H(x)=0
\]
for every nonconstant homogeneous harmonic polynomial \(H\) of degree at most
\(t\).

Delsarte, Goethals, and Seidel \cite{DGS1977} proved a Fisher-type lower
bound for spherical designs: if \(X\subset S^{n-1}\) is a spherical
\(t\)-design, then
\begin{equation}\label{eq:DGS-lower}
  |X|\ge
  \begin{cases}
    \displaystyle
    \binom{n+e-1}{e}+\binom{n+e-2}{e-1}
      & t=2e,\\[1.2em]
    \displaystyle
    2\binom{n+e-1}{e}
      & t=2e+1.
  \end{cases}
\end{equation}
A design attaining equality in \eqref{eq:DGS-lower} is called \emph{tight}.

We use the normalized Gegenbauer polynomials \(Q_i^{(n)}(u)\) from
Delsarte--Goethals--Seidel, which are characterized by
\begin{gather*}
  Q_0^{(n)}(u)=1,\qquad Q_1^{(n)}(u)=nu,\\
  Q_i^{(n)}(1)=\dim \Harm_i(\R^n)
  =
  \binom{n+i-1}{i}-\binom{n+i-3}{i-2}.
\end{gather*}
For \(m\ge0\), define
\begin{equation}\label{eq:Cdef}
  \cC_m^{(n)}(u):=
  Q_m^{(n)}(u)+Q_{m-2}^{(n)}(u)+Q_{m-4}^{(n)}(u)+\cdots,
\end{equation}
where the sum stops at \(Q_0^{(n)}\) or \(Q_1^{(n)}\) according to the parity
of \(m\).\footnote{The calligraphic notation \(\cC_m^{(n)}\) is reserved for
these cumulative DGS sums and is distinct from the classical Gegenbauer family
\(C_m^{\lambda}\).} Then
\begin{equation}\label{eq:Cvalue}
  \cC_m^{(n)}(1)=\binom{n+m-1}{m}.
\end{equation}

We use the following standard antipodal, or \emph{line-system}, form of the
Delsarte--Goethals--Seidel absolute bound.  It is the absolute-bound part and
equality case of \cite[Theorem 6.8]{DGS1977}, equivalently the line-system
bound of \cite{DGS1975,Koornwinder1976}.

\begin{theorem}[Delsarte--Goethals--Seidel]\label{thm:dgs-equality}
Let \(X\subset S^{n-1}\), \(n\ge2\), be a finite antipodal set, and put
\[
  A(X)=\{\langle x,y\rangle:x,y\in X,\ x\ne y\},\qquad s=|A(X)|.
\]
Then we have
\[
  |X|\le 2\,\cC_{s-1}^{(n)}(1)
  =
  2\binom{n+s-2}{s-1}.
\]
If equality holds, then \(X\) is a tight spherical \((2s-1)\)-design.
Moreover, if
\[
  F_X(u):=\prod_{\alpha\in A(X)}\frac{u-\alpha}{1-\alpha},
\]
then we have
\[
  |X|F_X(u)=(1+u)\cC_{s-1}^{(n)}(u).
\]
\end{theorem}

\begin{remark}
In the application to the norm-\(k\) shell of an integral lattice,
\[
  A(\Xk(\cL))\subseteq
  \{-1\}\cup
  \left\{\frac{j}{k}:-k<j<k,\ j\in\Z\right\},
\]
which is a set of size \(2k\).  Setting \(s=|A(\Xk(\cL))|\le 2k\) and using
the fact that \(\cC_m^{(n)}(1)=\binom{n+m-1}{m}\) is increasing in \(m\),
\cref{thm:dgs-equality} gives
\[
  |\Xk(\cL)|\le 2\cC_{s-1}^{(n)}(1)\le 2\cC_{2k-1}^{(n)}(1)
  =
  2\binom{n+2k-2}{2k-1},
\]
which is exactly \eqref{eq:RSDbound}.
\end{remark}

\subsection{Classification input for tight spherical designs}

We use the following nonexistence theorem for certain tight spherical designs,
due to Bannai--Damerell \cite{BD1979,BD1980} and Bannai \cite{Bannai1979},
building on the Delsarte--Goethals--Seidel work \cite{DGS1977}.

\begin{theorem}[Bannai--Damerell and Bannai]\label{thm:BD}
Let \(n\ge3\).  If a tight spherical \(t\)-design exists in \(S^{n-1}\) and
\(t\ge4\), then
\[
  t\in\{4,5,7,11\}.
\]
Equivalently, in Euclidean dimension at least \(3\), no tight spherical
\(t\)-designs exist for \(t\ge4\) except possibly in strengths
\(4\), \(5\), \(7,\) and \(11\).
\end{theorem}

\begin{remark}
The proof of \cref{thm:main} does not require a classification of all tight
spherical \(7\)-designs.  The integrality filter in
\cref{sec:integrality-filter} directly forces the \(k=2\) equality case to
have \(n=8\), after which the lattice shell is recognized as the \(E_8\) root
system.
\end{remark}

\section{Sharpness Forces Design Rigidity}\label{sec:equality-dgs}

The bound \eqref{eq:RSDbound} arises from \cref{thm:dgs-equality} applied to
the antipodal normalized shell, together with a monotonicity step.  Equality
in \eqref{eq:RSDbound} therefore forces both saturation of the DGS bound and
exhaustion of the allowed inner-product set, which together yield the
design and annihilator structure recorded in the proposition below.

\begin{proposition}\label{prop:equality-consequences}
Let \(n\ge2\), let \(\cL\subset\R^n\) be an integral lattice, and let
\(k\ge1\).  Suppose equality holds in \eqref{eq:RSDbound}.  Put
\[
  X=\Xk(\cL).
\]
Then:
\begin{enumerate}
\item We have\label{eqcons:full-inner-products}
\[
  A(X)=
  \{-1\}\cup
  \left\{\frac{j}{k}:-k<j<k,\ j\in\Z\right\}.
\]
\item\label{eqcons:tight-design}
\(X\) is an antipodal tight spherical \((4k-1)\)-design.
\item\label{eqcons:annihilator}
With
\[
  F_X(u)=\prod_{\alpha\in A(X)}\frac{u-\alpha}{1-\alpha},
\]
we have
\begin{equation}\label{eq:lattice-annihilator}
  2\binom{n+2k-2}{2k-1}F_X(u)
  =
  (1+u)\cC_{2k-1}^{(n)}(u).
\end{equation}
\item\label{eqcons:root-filter}
The roots of \(\cC_{2k-1}^{(n)}\) are exactly
\begin{equation}\label{eq:root-filter-general}
  \left\{0,\pm\frac1k,\pm\frac2k,\ldots,\pm\frac{k-1}{k}\right\}.
\end{equation}
\end{enumerate}
\end{proposition}

\begin{proof}
By integrality,
\[
  A(X)\subseteq A_k:=
  \{-1\}\cup\left\{\frac{j}{k}:-k<j<k,\ j\in\Z\right\},
\]
a set of size \(2k\).  Set \(s:=|A(X)|\), so \(s\le 2k\).  Since \(X\) is
antipodal, \cref{thm:dgs-equality} applies and gives
\begin{equation}\label{eq:DGS-applied-with-s}
  |X|\le 2\cC_{s-1}^{(n)}(1)
  =
  2\binom{n+s-2}{s-1}.
\end{equation}
By \eqref{eq:Cvalue}, \(\cC_m^{(n)}(1)=\binom{n+m-1}{m}\) is strictly
increasing in \(m\) for \(n\ge2\), so
\begin{equation}\label{eq:monotonicity-step}
  2\binom{n+s-2}{s-1}
  \le
  2\binom{n+2k-2}{2k-1},
\end{equation}
with equality if and only if \(s=2k\).  Chaining \eqref{eq:DGS-applied-with-s}
and \eqref{eq:monotonicity-step} reproduces \eqref{eq:RSDbound}.

Now assume that equality holds in \eqref{eq:RSDbound}.  Then both
\eqref{eq:DGS-applied-with-s} and \eqref{eq:monotonicity-step} are equalities.
Equality in \eqref{eq:monotonicity-step} forces \(s=2k\); together with
\(A(X)\subseteq A_k\) and \(|A_k|=2k\), this gives \(A(X)=A_k\), proving
\cref{eqcons:full-inner-products}.

With \(s=2k\), equality in \eqref{eq:DGS-applied-with-s} triggers the equality
statement of \cref{thm:dgs-equality}: \(X\) is an antipodal tight spherical
\((2s-1)=(4k-1)\)-design, which is \cref{eqcons:tight-design}, and
\[
  |X|\,F_X(u)=(1+u)\,\cC_{2k-1}^{(n)}(u).
\]
Substituting \(|X|=2\binom{n+2k-2}{2k-1}\) yields
\eqref{eq:lattice-annihilator}, which is \cref{eqcons:annihilator}.

Finally, \(F_X(u)\) is a polynomial of degree \(2k\) whose roots are precisely
the elements of \(A(X)\); in particular \(-1\) is one of them.  The right side
of \eqref{eq:lattice-annihilator} has \(-1\) as a simple root from the factor
\(1+u\), so the remaining roots---those of \(\cC_{2k-1}^{(n)}\)---coincide
with
\[
  A(X)\setminus\{-1\}
  =
  \left\{0,\pm\frac1k,\pm\frac2k,\ldots,\pm\frac{k-1}{k}\right\},
\]
which is \cref{eqcons:root-filter}.
\end{proof}

\begin{corollary}\label{cor:design-strength}
For \(n\ge2\), equality in \eqref{eq:RSDbound} at level \(k\) forces an
antipodal tight spherical \((4k-1)\)-design of size
\[
  2\binom{n+2k-2}{2k-1}.
\]
In particular, for \(k=2\) the forced design strength is \(7\), and for
\(k=3\) the forced design strength is \(11\).
\end{corollary}

\begin{proof}
This is \cref{eqcons:tight-design} of \cref{prop:equality-consequences}.
\end{proof}

\section{The Integrality Filter}\label{sec:integrality-filter}

The root condition \eqref{eq:root-filter-general} is already decisive for
small \(k\).  We need explicit formulas only for
\(\cC_1^{(n)}\), \(\cC_3^{(n)}\), and \(\cC_5^{(n)}\).

\begin{lemma}\label{lem:C135}
For \(n\ge2\),
\begin{align}
  \cC_1^{(n)}(u)
    &=nu, \label{eq:C1}\\
  \cC_3^{(n)}(u)
    &=\frac{n(n+2)}{6}\,u\bigl((n+4)u^2-3\bigr), \label{eq:C3}\\
  \cC_5^{(n)}(u)
    &=\frac{n(n+2)(n+4)}{120}\,
      u\bigl((n+6)(n+8)u^4-10(n+6)u^2+15\bigr). \label{eq:C5}
\end{align}
\end{lemma}

\begin{proof}
In the DGS normalization \cite[p.~365]{DGS1977}, the first odd normalized
Gegenbauer polynomials are
\[
  Q_1^{(n)}(u)=nu,
\]
\[
  6Q_3^{(n)}(u)=n(n+4)\bigl((n+2)u^3-3u\bigr),
\]
and
\[
  120Q_5^{(n)}(u)
  =
  n(n+2)(n+8)
  \bigl((n+4)(n+6)u^5-10(n+4)u^3+15u\bigr).
\]
By definition,
\[
  \cC_1^{(n)}=Q_1^{(n)},\qquad
  \cC_3^{(n)}=Q_3^{(n)}+Q_1^{(n)},\qquad
  \cC_5^{(n)}=Q_5^{(n)}+Q_3^{(n)}+Q_1^{(n)}.
\]
Substitution and simplification give \eqref{eq:C1}, \eqref{eq:C3}, and
\eqref{eq:C5}.
\end{proof}

\begin{proposition}[The \(k=2\) and \(k=3\) filters]\label{prop:k2k3-filter}
Let \(n\ge2\) and let \(\cL\subset\R^n\) be integral.
\begin{enumerate}
\item\label{kfilter:k-two}
If equality holds in \eqref{eq:RSDbound} for \(k=2\), then \(n=8\).
\item\label{kfilter:k-three}
Equality never holds in \eqref{eq:RSDbound} for \(k=3\).
\end{enumerate}
\end{proposition}

\begin{proof}
Suppose \(k=2\).  By \cref{eqcons:root-filter} of
\cref{prop:equality-consequences}, the roots of \(\cC_3^{(n)}\) must be
\[
  \left\{-\frac12,0,\frac12\right\}.
\]
By \eqref{eq:C3}, the nonzero roots of \(\cC_3^{(n)}\) are
\[
  \pm \sqrt{\frac{3}{n+4}}.
\]
Therefore
\[
  \sqrt{\frac{3}{n+4}}=\frac12,
\]
so \(n+4=12\), and hence \(n=8\).  This proves \cref{kfilter:k-two}.

Now suppose \(k=3\).  By \cref{eqcons:root-filter} of
\cref{prop:equality-consequences}, the roots of \(\cC_5^{(n)}\) must be
\[
  \left\{-\frac23,-\frac13,0,\frac13,\frac23\right\}.
\]
Thus the two nonzero squared roots must be \(1/9\) and \(4/9\).  By
\eqref{eq:C5}, those squared roots are the two roots of
\[
  (n+6)(n+8)Y^2-10(n+6)Y+15=0.
\]
The sum-of-roots condition gives
\[
  \frac{10(n+6)}{(n+6)(n+8)}
  =
  \frac{10}{n+8}
  =
  \frac19+\frac49
  =
  \frac59.
\]
Hence \(n=10\).  But at \(n=10\), the product of the roots is
\[
  \frac{15}{(n+6)(n+8)}
  =
  \frac{15}{16\cdot18}
  =
  \frac{5}{96},
\]
whereas the required product is
\[
  \frac19\cdot\frac49=\frac4{81}.
\]
Since \(5/96\ne4/81\), no \(n\) satisfies both conditions.  Thus equality is
impossible for \(k=3\).
\end{proof}

We also need an elementary exclusion in dimension \(2\).

\begin{lemma}\label{lem:circle}
Let \(n=2\) and \(k\ge2\).  No full-rank integral lattice
\(\cL\subset\R^2\) attains equality in \eqref{eq:RSDbound}.
\end{lemma}

\begin{proof}
Assume equality holds and put \(X=\Xk(\cL)\subset S^1\).  By
\cref{eqcons:tight-design} of \cref{prop:equality-consequences}, \(X\) is a
tight spherical \((4k-1)\)-design with
\[
  |X|=2\binom{2+2k-2}{2k-1}=4k.
\]
We use the elementary classification of tight designs on the circle.  Identify
\(S^1\) with the complex numbers of absolute value \(1\), and write
\[
  X=\{z_1,\ldots,z_N\},
  \qquad N=4k.
\]
A spherical \(t\)-design on \(S^1\) satisfies
\[
  \sum_{i=1}^N z_i^m=0
  \qquad
  (1\le |m|\le t).
\]
Here \(t=4k-1=N-1\), so
\[
  \sum_{i=1}^N z_i^m=0
  \qquad
  (1\le m\le N-1).
\]
Newton's identities imply that the elementary symmetric functions
\(e_1,\ldots,e_{N-1}\) vanish.  Hence
\[
  \prod_{i=1}^N(z-z_i)=z^N-c
\]
for some \(c\in\mathbb C\).  Since all \(z_i\) lie on the unit circle,
\(|c|=1\), and \(X\) is the vertex set of a regular \(N\)-gon.

Adjacent vertices of this regular \(4k\)-gon have inner product
\begin{equation}\label{eq:adjacent-vertices}
  \cos\left(\frac{2\pi}{4k}\right)
  =
  \cos\left(\frac{\pi}{2k}\right).
\end{equation}
For \(k\ge2\),
\[
  \cos\left(\frac{\pi}{2k}\right)
  >
  1-\frac12\left(\frac{\pi}{2k}\right)^2
  =
  1-\frac{\pi^2}{8k^2}
  >
  1-\frac1k
  =
  \frac{k-1}{k},
\]
where the last inequality is equivalent to \(k>\pi^2/8\), hence holds for
all \(k\ge2\).  Therefore \eqref{eq:adjacent-vertices} lies strictly between
\((k-1)/k\) and \(1\).  But the nontrivial inner products in a normalized
norm-\(k\) lattice shell must lie in
\[
  \left\{0,\pm\frac1k,\pm\frac2k,\ldots,\pm\frac{k-1}{k}\right\};
\]
this contradiction proves the lemma.
\end{proof}

\begin{proposition}\label{prop:dimension-bound}
Let \(n\ge2\), \(k\ge2\), and let \(\cL\subset\R^n\) be a full-rank integral
lattice.  If equality holds in \eqref{eq:RSDbound}, then \(k=2\) and \(n=8\).
\end{proposition}

\begin{proof}
For \(n=2\), \cref{lem:circle} rules out equality for all \(k\ge2\).

Now suppose \(n\ge3\).  By \cref{cor:design-strength}, equality gives a tight
spherical \((4k-1)\)-design in \(S^{n-1}\).  If \(k\ge4\), then
\(4k-1\ge15\), contradicting \cref{thm:BD}.  The case \(k=3\) is ruled out by
\cref{kfilter:k-three} of \cref{prop:k2k3-filter}.  Hence \(k=2\), and
\cref{kfilter:k-two} of \cref{prop:k2k3-filter} then gives \(n=8\).
\end{proof}

\section{The Main Classification}\label{sec:classification-proof}

We now prove \cref{thm:main}.

First consider \(n=1\).  A full-rank one-dimensional integral lattice has the
form
\[
  \cL=a\Z
\]
with \(a>0\) and \(a^2\in\Z\).  Its squared norms are \(a^2m^2\), with
\(m\in\Z\).  Therefore
\[
  N_{=k}(\cL)=
  \begin{cases}
    2 & k=a^2m^2 \text{ for some }m\in\Z,\\
    0 & \text{otherwise}.
  \end{cases}
\]
The right-hand side of \eqref{eq:RSDbound} is
\[
  2\binom{1+2k-2}{2k-1}=2.
\]
Thus equality holds exactly when \(\cL\) represents \(k\).  This proves
\cref{main:rank-one}.

Assume henceforth that \(n\ge2\).

For \(k=1\), equality means
\[
  N_{=1}(\cL)=2n.
\]
If \(y,z\in\cL\) have squared norm \(1\), then integrality and
Cauchy--Schwarz give
\[
  \ip{y}{z}\in\{-1,0,1\}.
\]
Thus, up to sign, distinct norm-\(1\) vectors are mutually orthogonal.
Equality gives \(n\) orthonormal vectors
\[
  e_1,\ldots,e_n\in\cL.
\]
Hence, we have
\[
  \Z e_1+\cdots+\Z e_n\subseteq\cL.
\]
Conversely, if \(v\in\cL\), then
\[
  \ip{v}{e_i}\in\Z
  \qquad
  (1\le i\le n),
\]
so the coordinates of \(v\) in the orthonormal basis \(e_1,\ldots,e_n\) are
all integers.  Therefore
\[
  v\in \Z e_1+\cdots+\Z e_n,
\]
and hence
\[
  \cL=\Z e_1+\cdots+\Z e_n\cong\Z^n.
\]
This proves the necessity in \cref{main:k-one}.  Conversely, \(\Z^n\) plainly
attains equality for \(k=1\).  Thus \cref{main:k-one} is proved.

Now assume \(k\ge2\).  By \cref{prop:dimension-bound}, equality is possible
only when \(k=2\) and \(n=8\).  Let
\[
  \Phi:=\Sk(\cL).
\]
Then
\[
  |\Phi|=2\binom{8+2}{3}=240.
\]
Moreover, for \(\alpha,\beta\in\Phi\),
\[
  \ip{\alpha}{\beta}\in\{-2,-1,0,1,2\},
\]
with \(\ip{\alpha}{\beta}=2\) only when \(\alpha=\beta\), and
\(\ip{\alpha}{\beta}=-2\) only when \(\alpha=-\beta\).

We claim that \(\Phi\) is a reduced crystallographic root system.  It is
finite and closed under negation.  Its normalized image is a spherical
\(2\)-design by \cref{eqcons:tight-design} of
\cref{prop:equality-consequences}, so it spans \(\R^8\): if a nonzero linear
form vanished on \(\Phi\), then the square of that linear form would
contradict the second-moment identity for a \(2\)-design.

For \(\alpha,\beta\in\Phi\), the reflection in the hyperplane perpendicular to
\(\alpha\) is
\[
  s_\alpha(\beta)
  =
  \beta-\frac{2\ip{\beta}{\alpha}}{\ip{\alpha}{\alpha}}\alpha
  =
  \beta-\ip{\beta}{\alpha}\alpha,
\]
because \(\norm{\alpha}^2=2\).  If \(\ip{\beta}{\alpha}=0\), then
\(s_\alpha(\beta)=\beta\).  If \(\ip{\beta}{\alpha}=1\), then
\[
  \norm{\beta-\alpha}^2=2+2-2=2,
\]
and \(\beta-\alpha\in\cL\), so \(\beta-\alpha\in\Phi\).  If
\(\ip{\beta}{\alpha}=-1\), the same argument gives
\(\beta+\alpha\in\Phi\).  The cases \(\ip{\beta}{\alpha}=\pm2\) are the
trivial cases \(\beta=\pm\alpha\).  Thus \(\Phi\) is closed under the root
reflections.

Furthermore,
\[
  \frac{2\ip{\alpha}{\beta}}{\ip{\alpha}{\alpha}}
  =
  \ip{\alpha}{\beta}\in\Z,
\]
so \(\Phi\) is crystallographic.  Since all roots in \(\Phi\) have the same
squared norm \(2\), the root system is simply laced; and since the only scalar
multiples of a root in \(\Phi\) are \(\pm\alpha\), it is reduced.

The irreducible simply-laced crystallographic root systems have types
\(A_r\), \(D_r\), and \(E_6,E_7,E_8\), with root counts
\begin{gather*}
  |A_r|=r(r+1),\qquad |D_r|=2r(r-1),
  \\
  |E_6|=72,\qquad |E_7|=126,\qquad |E_8|=240.
\end{gather*}
A reducible root system has root count equal to the sum of the root counts of
its irreducible factors.  Among simply-laced crystallographic root systems of
total rank \(8\), the number of roots is maximized uniquely by \(E_8\).
Since \(|\Phi|=240\), it follows that \(\Phi\) has type \(E_8\).

Let
\[
  M=\langle\Phi\rangle_{\Z}.
\]
Then \(M\subseteq\cL\), and \(M\) is the \(E_8\) root lattice.  Since \(E_8\)
is unimodular \cite[Ch.~4]{CS1999}, we have \(M^*=M\).  Because \(\cL\) is
integral and contains \(M\), every \(v\in\cL\) pairs integrally with every
element of \(M\).  Hence
\[
  v\in M^*=M.
\]
Thus \(\cL\subseteq M\), and therefore
\[
  \cL=M\cong E_8;
\]
this proves the necessity in \cref{main:e-eight}.  Conversely, \(E_8\) attains
equality for \(k=2\), as we check in \cref{sec:verification}.  This proves
\cref{main:e-eight} and completes the proof of \cref{thm:main}.

\begin{corollary}[Shell-generated form]\label{cor:shell-generated}
Let \(\cL\subset\R^n\) be an integral lattice with
\(\Sk(\cL)\neq\varnothing\), let
\[
  M=\langle\Sk(\cL)\rangle_{\Z},
\]
and put \(r=\rank M\).  Suppose that the shell saturates the RSD bound in its
own linear span, i.e.,
\[
  |\Sk(\cL)|=2\binom{r+2k-2}{2k-1}.
\]
Then one of the following holds:
\begin{enumerate}
\item \(r=1\) and \(M\cong\sqrt{k}\Z\);
\item \(r\ge2\), \(k=1\), and \(M\cong\Z^r\);
\item \(r=8\), \(k=2\), and \(M\cong E_8\).
\end{enumerate}
Moreover, if \(\cL\) itself attains equality in \eqref{eq:RSDbound} in
dimension \(n\), then \(r=n\).
\end{corollary}

\begin{proof}
Since \(M\) is generated by \(\Sk(\cL)\), we have
\[
  \Sk(M)=\Sk(\cL).
\]
Viewing \(M\) as a full-rank integral lattice in its real span, the hypothesis
says exactly that \(M\) attains equality in the RSD bound in dimension \(r\).
Applying \cref{thm:main} to \(M\) gives the three listed possibilities.

Finally, suppose that \(\cL\) itself attains equality in dimension \(n\).
Then
\[
  2\binom{n+2k-2}{2k-1}
  =
  |\Sk(\cL)|
  =
  |\Sk(M)|
  \le
  2\binom{r+2k-2}{2k-1}.
\]
Since \(r\le n\) and the function
\[
  d\longmapsto \binom{d+2k-2}{2k-1}
\]
is strictly increasing for \(d\ge1\), we get \(r=n\).
\end{proof}

\section{Equality Cases and Comparison Examples}\label{sec:verification}

\subsection{The rank-\(1\) family}

Let \(\cL=a\Z\subset\R\) with \(a^2\in\Z_{>0}\).  If
\[
  k=a^2m^2
\]
for some \(m\in\Z_{>0}\), then
\[
  \Sk(\cL)=\{\pm am\},
\]
so we have
\[
  N_{=k}(\cL)=2;
\]
since
\[
  2\binom{1+2k-2}{2k-1}=2,
\]
equality holds.  If no such \(m\) exists, then \(N_{=k}(\cL)=0\), so equality
fails.

After replacing \(\cL\) by the sublattice generated by its norm-\(k\) shell,
the equality case becomes
\[
  \sqrt{k}\Z.
\]

\subsection{The case \((\Z^n,1)\)}

For the standard lattice \(\Z^n\),
\[
  \Sk(\Z^n)=\{\pm e_1,\ldots,\pm e_n\},
\]
so we have
\[
  N_{=1}(\Z^n)=2n.
\]
The bound gives
\[
  2\binom{n+2-2}{1}=2n,
\]
so \((\Z^n,1)\) attains equality.

\subsection{The case \((E_8,2)\)}

The \(E_8\) root lattice is even unimodular of rank \(8\), and its roots have
squared norm \(2\).  The \(E_8\) root system has \(240\) roots.  Therefore
\[
  N_{=2}(E_8)=240.
\]
The right-hand side of \eqref{eq:RSDbound} for \(n=8\), \(k=2\), is
\[
  2\binom{8+4-2}{3}
  =
  2\binom{10}{3}
  =
  240.
\]
Thus \((E_8,2)\) attains equality.

The normalized roots of \(E_8\) form the tight spherical \(7\)-design in
\(S^7\) \cite[Example~8.4]{DGS1977} (see also \cite{Venkov2001}); their
nontrivial inner products are
\[
  -1,\quad -\frac12,\quad 0,\quad \frac12,
\]
which agrees with the \(k=2\) integrality filter:
\[
  \Roots\bigl(\cC_3^{(8)}\bigr)
  =
  \left\{-\frac12,0,\frac12\right\}.
\]

\subsection{The \(k=3\) circle comparison}

For \(n=2\) and \(k=3\), the bound is
\[
  2\binom{2+6-2}{5}
  =
  2\binom{6}{5}
  =
  12.
\]
Equality would force a regular \(12\)-gon on the circle by the argument in
\cref{lem:circle}; its adjacent inner product is
\[
  \cos\left(\frac{\pi}{6}\right)=\frac{\sqrt3}{2},
\]
which is not in \((1/3)\Z\).  Thus equality is impossible in dimension \(2\).

In all dimensions \(n\ge2\), \cref{kfilter:k-three} of
\cref{prop:k2k3-filter} rules out equality at \(k=3\) by the DGS annihilator
roots.

\subsection{The Leech lattice comparison}

The Leech lattice \(\Lambda_{24}\) is even unimodular of rank \(24\), has
minimum squared norm \(4\), and has
\[
  N_{=4}(\Lambda_{24})=196{,}560
\]
minimal vectors \cite{Leech1967}.  The bound \eqref{eq:RSDbound} for
\(n=24\), \(k=4\), is
\[
  2\binom{24+8-2}{7}
  =
  2\binom{30}{7}
  =
  4{,}071{,}600.
\]
Thus the Leech lattice is far from equality in \eqref{eq:RSDbound} at
\(k=4\).

The normalized minimal vectors of the Leech lattice form a tight spherical
\(11\)-design in \(S^{23}\).  This does not contradict the equality
classification above: equality in \eqref{eq:RSDbound} at level \(k=4\) would
force the stronger condition of a tight spherical \(15\)-design.  Indeed, the
normalized Leech minimal-vector inner products are
\[
  -1,\quad -\frac12,\quad -\frac14,\quad 0,\quad \frac14,\quad \frac12,
\]
which correspond to the roots prescribed by \(\cC_5^{(24)}\), together with
the antipodal value \(-1\).  Equality in \eqref{eq:RSDbound} at \(k=4\) would
instead require the roots of \(\cC_7^{(24)}\) to be
\[
  0,\quad \pm\frac14,\quad \pm\frac12,\quad \pm\frac34,
\]
and would force a tight spherical \(15\)-design, which is ruled out in
dimension at least \(3\) by \cref{thm:BD}.

\section{Concluding Remarks}\label{sec:remarks}

\subsection{The equality list}

In shell-generated form, the complete equality list is
\[
  (\sqrt{k}\Z,k)\quad (k\ge1),
  \qquad
  (\Z^n,1)\quad (n\ge2),
  \qquad
  (E_8,2).
\]
The rank-\(1\) family appears because the DGS shell bound equals \(2\) in
dimension \(1\) for every \(k\), and any one-dimensional integral lattice
representing \(k\) saturates it.

\subsection{Why equality is more rigid than the bound}

The upper bound \eqref{eq:RSDbound} uses only the coarse list of inner
products allowed by integrality.  It does not require all those inner products
to occur, and it does not use their multiplicities.  Equality, however, forces
all arithmetically allowed inner products to occur and imposes the DGS
annihilator identity.

Thus the equality problem is much more rigid than the inequality
itself.  For \(k=2\), for example, the polynomial identity alone gives
\[
  \sqrt{\frac{3}{n+4}}=\frac12,
\]
and hence \(n=8\).  The lattice structure then reconstructs the \(E_8\) root
system.

\subsection{Relation to the reverse Minkowski program}

The theorem of Regev and Stephens-Davidowitz \cite{RSD2026} gives a
reverse Minkowski--type shell bound for integral lattices, in the broader
context of reverse Minkowski questions for stable lattices
\cite{RSD2024,Reg22}.  Our equality classification shows that the DGS shell
bound is almost never exactly sharp in rank at least \(2\); indeed, for every
\(k\ge3\), no full-rank integral lattice of dimension \(n\ge2\) attains it.

This leaves several quantitative questions.

First, one can ask for the true maximum of \(N_{=k}(\cL)\) among integral
lattices in fixed dimension \(n\), especially for fixed \(k\ge3\).  The DGS
bound \eqref{eq:RSDbound} has order \(n^{2k-1}\) for fixed \(k\), whereas the
standard cubic lattice has only \(\Theta(n^k)\) vectors of squared norm \(k\).
Thus the equality classification does not determine the true extremal
asymptotics; it only rules out exact saturation of the DGS bound.

Second, one can ask for stability.  If \(N_{=k}(\cL)\) is close to the DGS
bound, must the normalized shell be close to a tight design, or close to one
of the known sharp configurations?  The proof here uses exact DGS equality;
a robust version would require quantitative stability for the relevant
linear-programming bound.

Third, the Gaussian mass question remains compelling.  The conjecture asks
whether
\[
  \sum_{y\in \cL} e^{-\tau\norm{y}^2}
  \le
  \sum_{z\in \Z^n} e^{-\tau\norm{z}^2}
\]
holds for all stable lattices \(\cL\subset\R^n\) and all \(\tau>0\).  Known
partial results include the analogous Epstein zeta inequality of Eisenberg,
Regev, and Stephens-Davidowitz \cite{ERSD2022}, which gives
\[
  \zeta(\cL,s)\le \zeta(\Z^n,s)
  \qquad (s>n/2),
\]
with equality only for \(\cL\) isometric to \(\Z^n\).  Regev and
Stephens-Davidowitz \cite{RSD2024} also resolved the Gaussian mass question
in extreme parameter regimes.  The principal obstruction in the intermediate
regime is the existence of local maxima of Gaussian mass in high dimensions,
as established by Heimendahl, Marafioti, Thiemeyer, Vallentin, and
Zimmermann \cite{HMTVZ2023}.  The present equality classification operates at
the level of individual shells rather than generating functions, so the rarity
of exact DGS saturation neither implies nor refutes the Gaussian mass
conjecture.

\end{document}